\documentclass[12pt,a4paper]{amsart}

\usepackage{amsmath,amsthm,amsfonts,amssymb,array,amscd}
\usepackage{amsmath,geometry,amssymb,amsfonts,amsthm,graphicx,enumerate,latexsym,tabularx,amscd}

\usepackage{refcheck} 
\norefnames
\nocitenames
\usepackage{geometry}
\usepackage{amsthm}
 \newtheorem{thm}{Theorem}[section]
 \newtheorem{lem}[thm]{Lemma}
\theoremstyle{definition}
 \newtheorem{exm}[thm]{Example}
 \newtheorem{dfn}[thm]{Definition}
 \newtheorem{rem}[thm]{Remark}
 \numberwithin{equation}{section}
\theoremstyle{definition}
\theoremstyle{remark}
 \numberwithin{equation}{section}

\renewcommand{\ge}{\geqslant}\renewcommand{\geq}{\geqslant}

\usepackage{tikz}
\usepackage[pdftex]{hyperref}
\usetikzlibrary{matrix,arrows}

\newcommand{\bbC}{\mathbb{C}}
\newcommand{\bbF}{\mathbb{F}}

\newcommand{\bbQ}{\mathbb{Q}}

\renewcommand{\and}{\quad \mbox{and} \quad}  

\renewcommand{\ge}{\geqslant}\renewcommand{\geq}{\geqslant}

\title{Langlands' Lambda function for quadratic tamely ramified extensions}


\author[Biswas]{\bfseries Sazzad Ali Biswas}

\address{
Chennai Mathematical Institute\\ 
H1, Sipcot It Park, Siruseri  \\ 
Kelambakkam, 603103\\
India}
\email{sabiswas@cmi.ac.in, sazzad.jumath@gmail.com}

\thanks{The  author is partially supported by Infosys Foundation, India}

\begin{document}

\vspace{18mm}
\setcounter{page}{1}
\thispagestyle{empty}

\begin{abstract}
 Let $K/F$ be a quadratic tamely ramified extension of a non-Archimedean local field $F$ of characteristic zero.
In this paper, 
we give an explicit formula for Langlands' lambda function $\lambda_{K/F}$. 
 
\end{abstract}

\maketitle

\section{\textbf{Introduction}}

Let $K/F$ be a finite subextension (but need not be Galois) in $\overline{F}/F$, where $\overline{F}$ is an algebraic
closure of a non-Archimedean
 local field $F$ of characteristic zero. Let $\psi$ be a nontrivial additive character of $F$. Then Langlands's lambda 
 function (or simply $\lambda$-function) (cf. \cite{RL}) of the extension $K/F$ is:
$$\lambda_{K/F}(\psi):=W(\text{Ind}_{G_K}^{G_F}(1_K),\psi),$$
where $1_K$ is the trivial representation of $G_K:=\text{Gal}(\overline{F}/K)$. Here $W$ denotes for local constant (or
epsilon factor) (cf. \cite{JT1}). We also can define the $\lambda$-function
via Deligne's constant $c(\rho):=\frac{W(\rho)}{W(\det(\rho))}$, where $\rho$ is a finite dimensional 
representation of $G_F$ and 
$\det(\rho)$ is the determinant of $\rho$.

Langlands has shown (cf. Theorem 1 on p. 105 of \cite{JT1}) that the local constants are {\bf weakly} extendible functions.
Therefore, to compute the local constant of any induced local Galois representation, we have to compute 
the $\lambda$-function explicitly.
Since the local Langlands correspondence preserves local constants, the explicit computation of local constants is an 
important part of the Langlands program. In the automorphic side of the local Langlands correspondence, we have local converse
theorem (cf. \cite{JWC}, \cite{GH}), but so far we do not have any such converse theorem in the 
Galois side {\bf because} explicit computation of $\lambda$-functions (hence epsilon factors of local Galois representations)
are 
not known. In the previous paper \cite{SABNT} the author gave an explicit computation of the lambda function for any 
tamely ramified Galois extension assuming the
computation of the lambda function for a tamely ramified quadratic extension.  In this paper we provide a formula for
the tamely ramified quadratic case, thus completing the work in \cite{SABNT}.

We should mention that in \cite{TS} Saito has computed the $\lambda$-function for an arbitrary extension 
assuming the residual characteristic
of the base field is not equal to 2 (cf. Theorem on p. 508 of \cite{TS}). In Theorem II 2B on p. 508
of \cite{TS}, when ramification
index is {\bf even}, Saito has computed the lambda functions for even degree extensions via 
the {\bf Legendre symbol} and {\bf Hilbert symbol}.

In this paper, we also compute this $\lambda$-functions for quadratic tamely ramified
extensions. In our computation, we use the classical quadratic
{\bf Gauss sums} and
these computations are different from the Saito's result and explicit.
The main idea for tamely ramified quadratic extension case is to reduce the $\lambda$-functions computation
to the classical quadratic Gauss sums computations.


In this paper, as mentioned above, we compute lambda functions for quadratic tamely ramified extensions explicitly. 
This, together with the work in \cite{SABNT},
yields an explicit computation of lambda functions for all tamely ramified extensions.

We now state the main theorem of this paper.


\begin{thm}\label{Theorem 3.21}
 Let $K$ be a tamely ramified quadratic extension of $F/\bbQ_p$ with $q_F=p^s$. Let $\psi_F$ be the canonical additive 
 character of $F$.
 Let $c\in F^\times$ with $-1=\nu_F(c)+d_{F/\bbQ_p}$, and $c'=\frac{c}{\text{Tr}_{F/F_0}(pc)}$, where $F_0/\bbQ_p$ is the 
 maximal unramified
 extension in $F/\bbQ_p$. Let $\psi_{-1}=c'\cdot\psi_F$, then
 \begin{equation*}
  \lambda_{K/F}(\psi_F)=\Delta_{K/F}(c')\cdot\lambda_{K/F}(\psi_{-1}),
 \end{equation*}
where 
 \begin{equation*}
 \lambda_{K/F}(\psi_{-1})=\begin{cases}
                                               (-1)^{s-1} & \text{if $p\equiv 1\pmod{4}$}\\
                                                  (-1)^{s-1}i^{s} & \text{if $p\equiv 3\pmod{4}$}.
                                            \end{cases}
\end{equation*}
If we take $c=\pi_{F}^{-1-d_{F/\bbQ_p}}$, where $\pi_F$ is a norm for $K/F$, then 
\begin{equation}
 \Delta_{K/F}(c')=\begin{cases}
                   1 & \text{if $\overline{\text{Tr}_{F/F_0}(pc)}\in k_{F_0}^{\times}=k_{F}^{\times}$ is a square},\\
                   -1 & \text{if $\overline{\text{Tr}_{F/F_0}(pc)}\in k_{F_0}^{\times}=k_{F}^{\times}$ is not a square}.
                  \end{cases}
\end{equation}
Here "overline" stands for modulo $P_{F_0}$ and $c'\cdot \psi_F(x):=\psi_F(c'x)$ for all $x\in F$.
\end{thm}
\begin{rem}\label{Remark 1.2}
 
But in general, computation of $\lambda_{K/F}$, where $K/F$ is a wildly
ramified quadratic extension, seems subtle. When $F=\bbQ_2,$ in Example 3.4.14, pp. 60-63 of \cite{SABT}, we 
have explicit computation for $\lambda_{K/\bbQ_2}$. 
In \cite{MW}, one also can find some particular cases (cf. on pp. 252-255 of \cite{MW}). 
But if $F/\bbQ_2$ is an arbitrary finite extension and $K/F$ is {\bf quadratic} extension, then computation of 
$\lambda_{K/F}$ is still {\bf open.} And if $K$ is an abelian extension with $N_{K/F}(K^\times)={F^\times}^2$, we have the 
following theorem.
\end{rem}

\begin{thm}\label{Theorem 3.26}
 Let $F$ be an extension of $\bbQ_2$. Let $K$ be the abelian extension for which $N_{K/F}(K^\times)={F^\times}^2$.
 Then $\lambda_{K/F}=1$.
\end{thm}

On pp. 7-8, Theorem \ref{Theorem 3.21} and Theorem \ref{Theorem 3.26} are proven.

\section{\textbf{Notations and Preliminaries}}

Let $F$  be a non-Archimedean local field of characteristic zero,
i.e., a finite extension of the field $\mathbb{Q}_p$ (field of $p$-adic numbers),
where $p$ is a prime.
Let $O_F$ be the 
ring of integers in the local field $F$ and $P_F=\pi_F O_F$ is the unique prime ideal in $O_F$ 
and $\pi_F$ is a uniformizer, i.e., an element in $P_F$ whose valuation is one, i.e.,
 $\nu_F(\pi_F)=1$. Let $q_F$ be the cardinality of the residue field $k_F$ of $F$.
Let $U_F=O_F-P_F$ be the group of units in $O_F$.
Let $P_{F}^{i}=\{x\in F:\nu_F(x)\geq i\}$ and for $i\geq 0$ define $U_{F}^i=1+P_{F}^{i}$
(with proviso $U_{F}^{0}=U_F=O_{F}^{\times}$).
We also consider that $a(\chi)$ is the conductor of 
 nontrivial character $\chi: F^\times\to \mathbb{C}^\times$, i.e., $a(\chi)$ is the smallest integer $m\geq 0$ such 
 that $\chi$ is trivial
 on $U_{F}^{m}$. We say $\chi$ is unramified if the conductor of $\chi$ is zero and otherwise ramified.
Throughout the paper, when $K/F$
is unramified we choose uniformizers $\pi_K=\pi_F$. And when $K/F$ is ramified (both tame and wild) we choose
uniformizers $\pi_F=N_{K/F}(\pi_K)$, where $N_{K/F}$ is the norm map from $K^\times$ to $F^\times$.
In this paper $\Delta_{K/F}:=\det(\text{Ind}_{K/F}(1))$.

The conductor of any nontrivial additive character $\psi$ of the field $F$ is an integer $n(\psi)$ if $\psi$ is trivial
on $P_{F}^{-n(\psi)}$, but nontrivial on $P_{F}^{-n(\psi)-1}$. 
 
 

 

\subsection{Local constant formula for character}

For a nontrivial multiplicative character $\chi$ of $F^\times$ and nontrivial additive character $\psi$ of $F$, we have 
(cf. \cite{JT1}, p. 94):
\begin{equation}\label{eqn 2.5}
 W(\chi,\psi)=\chi(c)q_F^{-a(\chi)/2}\sum_{x\in\frac{U_F}{U_{F}^{a(\chi)}}}\chi^{-1}(x)\psi(x/c),
\end{equation}
where $c=\pi_{F}^{a(\chi)+n(\psi)}$. 


\begin{dfn}[\textbf{Canonical additive character}] \label{Definition of canonical additive character}

We define the non trivial additive character of $F$, $\psi_F:F\to\mathbb{C}^\times$ as the composition of the following 
four maps:
\begin{center}
 $F\xrightarrow{\mathrm{Tr}_{F/\mathbb{Q}_p}}\mathbb{Q}_p\xrightarrow{\alpha}\mathbb{Q}_p/\mathbb{Z}_p
 \xrightarrow{\beta}\mathbb{Q}/\mathbb{Z}\xrightarrow{\gamma}\mathbb{C}^\times$,
\end{center}
where
\begin{enumerate}
 \item $\mathrm{Tr}_{F/\mathbb{Q}_p}$ is the trace from $F$ to $\mathbb{Q}_p$,
 \item $\alpha$ is the canonical surjection map,
 \item $\beta$ is the canonical injection which maps $\mathbb{Q}_p/\mathbb{Z}_p$ onto the $p$-component of the 
 divisible group $\mathbb{Q}/\mathbb{Z}$ and 
 \item $\gamma$ is the exponential map $x\mapsto e^{2\pi i x}$, where $i=\sqrt{-1}$.
\end{enumerate}
For every $x\in\mathbb{Q}_p$, there is a rational $r$, uniquely determined modulo $1$, such that $x-r\in\mathbb{Z}_p$.
Then $\psi_{\bbQ_p}(x)=\psi_{\bbQ_p}(r)=e^{2\pi i r}$.
The nontrivial additive character  $\psi_F=\psi_{\bbQ_p}\circ \rm{Tr}_{F/\bbQ_p}$ of $F$ 
is called the \textbf{canonical additive character} (cf. \cite{JT1}, p. 92).
\end{dfn}

\subsection{Classical Gauss sums}

Let $k_q$ be a finite field. Let $p$ be the characteristic of $k_q$; then the prime  field 
contained in $k_q$ is $k_p$. 
The structure of the {\bf canonical} additive character $\psi_q$ of $k_q$ is the same as the structure 
of the canonical (see the definition \ref{Definition of canonical additive character})  character
$\psi_F$, namely {\bf it comes by trace} from the canonical character of the base field, i.e., 
\begin{center}
 $\psi_q=\psi_p\circ \text{Tr}_{k_q/k_p}$,
\end{center}
where 
\begin{center}
 $\psi_p(x):=e^{\frac{2\pi i x}{p}}$ \hspace{.3cm} for all $x\in k_p$.
\end{center}

\textbf{Gauss sums:} Let $\chi, \psi$ be a multiplicative and an additive character respectively of $k_q$. 
Then the Gauss sum $G(\chi,\psi)$ is 
defined
by 
\begin{equation}
 G(\chi,\psi)=\sum_{x\in k_{q}^{\times}}\chi(x)\psi(x).
\end{equation}
For computation of $\lambda_{K/F}$, where $K/F$ is a tamely ramified quadratic extension, we will use the following 
theorem.
\begin{thm}[\cite{LN}, p. 199, Theorem 5.15]\label{Theorem 2.7}
Let $k_q$ be a finite field with $q=p^s$, where $p$ is an odd prime and $s\in\mathbb{N}$. Let $\chi$ be the quadratic character of 
$k_q$ and let $\psi$ be the canonical additive character of $k_q$. Then
\begin{equation}
 G(\chi,\psi)=\begin{cases}
               (-1)^{s-1}q^{\frac{1}{2}} & \text{if $p\equiv 1\pmod{4}$},\\
               (-1)^{s-1}i^sq^{\frac{1}{2}} & \text{if $p\equiv 3\pmod{4}$}.
              \end{cases}
\end{equation}
\end{thm}

\section{\textbf{Explicit computation of $\lambda_{K/F}$, where $K/F$ is a quadratic extension}}

Let $K/F$ be a quadratic extension of the field $F/\bbQ_p$. Let $G=\mathrm{Gal}(K/F)$
be the Galois group of the extension $K/F$. Let $t$ be the \textbf{ramification break or jump} (cf. \cite{JPS}) of 
the Galois group
$G$ (or of the extension $K/F$). Then it can be
proved that the conductor of 
$\omega_{K/F}$ (the quadratic character of $F^\times$ associated to $K$ by class field theory) is
$t+1$. 
When $K/F$ is unramified we have $t=-1$, therefore the conductor of a quadratic character $\omega_{K/F}$ of $F^\times$
is zero, i.e., 
$\omega_{K/F}$ is unramified.
And when $K/F$ is tamely ramified we have $t=0$, then $a(\omega_{K/F})=1$.
In the wildly ramified case (which occurs if $p=2$) it can be proved that $a(\omega_{K/F})=t+1$ is,
{\bf up to the exceptional case $t=2\cdot e_{F/\bbQ_2}$}, always an \textbf{even number} which 
can be seen by the filtration of $F^\times$ (cf. p. 50 of \cite{SABT}).

 \subsection{\textbf{Computation of $\lambda_{K/F}$, where $K/F$ is a tamely ramified quadratic extension}}

The existence of a tamely ramified quadratic character $\chi$ (which is not unramified)
 of a local field $F$ implies $p\ne2$ for the residue characteristic. Then
 \begin{center}
  $F^\times/{F^\times}^2\cong V$
 \end{center}
is isomorphic to Klein's $4$-group. So we have only $3$ nontrivial quadratic characters in that case, 
corresponding to $3$
quadratic extensions $K/F$. One is unramified and other two are ramified. The unramified case is well settled.
The two ramified quadratic characters determine two different 
quadratic ramified extensions of $F$.\\
In the ramified case we have $a(\chi)=1$ because it is tame,
and we take $\psi$ of conductor $-1$. Then we have $a(\chi)+n(\psi)=0$ and therefore in the formula of $W(\chi,\psi)$ 
(cf. equation (\ref{eqn 2.5})) we can take $c=1$. So we obtain:
\begin{equation}\label{eqn 3.33}
 W(\chi,\psi)=q_{F}^{-\frac{1}{2}}\sum_{x\in U_F/U_{F}^{1}}\chi^{-1}(x)\psi(x)=
 q_{F}^{-\frac{1}{2}}\sum_{x\in k_{F}^{\times}}\chi'^{-1}(x)\psi'(x),
\end{equation}
where $\chi'$ is the quadratic character of the residue field $k_{F}^{\times}$, and 
$\psi'$ is an additive character of $k_F$.
When $n(\psi)=-1$, we observe that both {\bf the ramified characters $\chi$ give the same $\chi'$, hence
the same 
$W(\chi,\psi)$}, 
because one is different from other by a quadratic unramified character twist.
To compute an explicit formula for $\lambda_{K/F}(\psi_{-1})$, where $K/F$ is a tamely ramified quadratic extension and 
$\psi_{-1}$ is an additive character of $F$ with conductor $-1$, we need to use 
{\bf classical quadratic Gauss sums}.


 Let $\psi_{-1}$ be an additive character of $F/\bbQ_p$ of conductor $-1$, i.e.,
 $\psi_{-1}:F/P_F\to\bbC^\times$. Now restrict 
$\psi_{-1}$ to $O_F$, it will be one of the characters $a\cdot\psi_{q_F}$, for some $a\in k_{q_F}^{\times}$ and usually it will not be 
$\psi_{q_F}$ itself.
Therefore, choosing $\psi_{-1}$ is very important and we have to choose $\psi_{-1}$ such a way that its restriction to $O_F$
is exactly $\psi_{q_F}$. Then we will be able to use the quadratic classical Gauss sum in
the $\lambda$-function computation.
We also know that there exists an element $c\in F^\times$ such that 
\begin{equation}\label{eqn 3.34}
 \psi_{-1}=c\cdot\psi_F
\end{equation}
induces the canonical character $\psi_{q_F}$ on the residue field $k_F$.

Now question is: {\bf Finding proper $c\in F^\times$ for which $\psi_{-1}|_{O_F}=c\cdot\psi_F|_{O_F}=\psi_{q_F}$}, i.e., 
the canonical character of the residue field $k_F$.

From the definition of conductor of the additive character $\psi_{-1}$ of $F$, we obtain from the construction (\ref{eqn 3.34})
\begin{equation}\label{eqn 3.35}
 -1=\nu_F(c)+n(\psi_F)=\nu_F(c)+d_{F/\bbQ_p},
\end{equation}
where $d_{F/\bbQ_p}$ is the exponent of the different $\mathcal{D}_{F/\bbQ_p}$.
In the next two lemmas we choose the proper $c$ for our requirement.

\begin{lem}\label{Lemma 3.20}
 Let $F/\bbQ_p$ be a local field and let $\psi_{-1}$ be an additive character of $F$ of conductor $-1$. 
 Let $\psi_F$ be the canonical
 character of $F$. Let $c\in F^\times$ be any element such that $-1=\nu_F(c)+d_{F/\bbQ_p}$, and 
 \begin{equation}\label{eqn 3.36}
  \text{Tr}_{F/F_0}(c)=\frac{1}{p},
 \end{equation}
where $F_0/\bbQ_p$ is the maximal unramified subextension in $F/\bbQ_p$. Then 
the restriction of $\psi_{-1}=c\cdot\psi_F$ to $O_F$ is the canonical character $\psi_{q_F}$ of the residue field $k_F$ of $F$.
\end{lem}
\begin{proof}
 Since $F_0/\bbQ_p$ is the maximal unramified subextension in $F/\bbQ_p$, we have $\pi_{F_0}=p$, and the residue fields of 
 $F$ and $F_0$ are 
 isomorphic, i.e., $k_{F_0}\cong k_{F}$, because $F/F_0$ is totally ramified extension.
 Then every element of $O_F/P_F$ can be considered as an element of $O_{F_0}/P_{F_0}$.
 Moreover, since $F_0/\bbQ_p$ is the maximal unramified extension, then from Proposition 2 of \cite{AW} on p. 140,
 for $x\in O_{F_0}$ we have 
 $$\rho_{p}(\text{Tr}_{F_0/\bbQ_p}(x))=\text{Tr}_{k_{F_0}/k_{\bbQ_p}}(\rho_0(x)),$$
 where $\rho_0,\,\rho_p$ are the canonical homomorphisms of $O_{F_0}$ onto $k_{F_0}$, and of $O_{\bbQ_p}$ onto $k_{\bbQ_p}$,
 respectively. Then 
 for $x\in k_{F_0}$ we can write
 \begin{equation}\label{eqn 3.50}
  \text{Tr}_{F_0/\bbQ_p}(x)=\text{Tr}_{k_{F_0}/k_{\bbQ_p}}(x).
 \end{equation}
Furthermore, since $F/F_0$ is totally ramified, we have $k_{F}=k_{F_0}$, then the trace map 
for the tower of the residue fields $k_{F}/k_{F_0}/k_{\bbQ_p}$ is:
\begin{equation}\label{eqn 3.51}
 \text{Tr}_{k_F/k_{\bbQ_p}}(x)=\text{Tr}_{k_{F_0}/k_{\bbQ_p}}\circ\text{Tr}_{k_F/k_{F_0}}(x)=\text{Tr}_{k_{F_0}/k_{\bbQ_p}}(x),
\end{equation}
for all $x\in k_F$. Then from the equations (\ref{eqn 3.50}) and (\ref{eqn 3.51}) we obtain 
\begin{equation}\label{eqn 3.52}
 \text{Tr}_{F_0/\bbQ_p}(x)=\text{Tr}_{k_F/k_{\bbQ_p}}(x)
\end{equation}
 for all $x\in k_F$.

Since the conductor of $\psi_{-1}$ is $-1$, for $x\in O_F/P_F(=O_{F_0}/P_{F_0}$ because $F/F_0$ is totally ramified)  we have  
\begin{align*}
  \psi_{-1}(x)
  &=c\cdot\psi_F(x)
  =\psi_{F}(cx)
  =\psi_{\bbQ_p}(\text{Tr}_{F/\bbQ_p}(cx))
  =\psi_{\bbQ_p}(\text{Tr}_{F_0/\bbQ_p}\circ\text{Tr}_{F/F_0}(cx))\\
  &=\psi_{\bbQ_p}(\text{Tr}_{F_0/\bbQ_p}(x\cdot\text{Tr}_{F/F_0}(c)))\\
  &=\psi_{\bbQ_p}(\text{Tr}_{F_0/\bbQ_p}(\frac{1}{p}x)), \quad\text{since $x\in O_F/P_F=O_{F_0}/P_{F_0}$ and $\text{Tr}_{F/F_0}(c)=\frac{1}{p}$}\\
  &=\psi_{\bbQ_p}(\frac{1}{p}\text{Tr}_{F_0/\bbQ_p}(x)), \quad\text{because $\frac{1}{p}\in\bbQ_p$}\\
 &=e^{\frac{2\pi i \text{Tr}_{F_0/\bbQ_p}(x)}{p}},\quad\text{because $\psi_{\bbQ_p}(x)=e^{2\pi i x}$}\\
&=e^{\frac{2\pi i\text{Tr}_{k_F/k_{\bbQ_p}}(x)}{p}},\quad \text{using equation $(\ref{eqn 3.52})$}\\
&=\psi_{q_F}(x).
 \end{align*}
This competes the lemma.

\end{proof}

The next step is to produce good elements $c$ more explicitly. By using Lemma \ref{Lemma 3.20}, 
in the next lemma we see more general choices of $c$.

\begin{lem}\label{Lemma 3.21}
 Let $F/\bbQ_p$ be a tamely ramified local field and let $\psi_{-1}$ be an additive character of $F$ 
 of conductor $-1$. Let $\psi_F$ be the canonical
 character of $F$. Let $F_0/\bbQ_p$ be the maximal unramified subextension in $F/\bbQ_p$.
 Let $c\in F^\times$ be any element such that $-1=\nu_F(c)+d_{F/\bbQ_p}$, then 
 \begin{center}
  $c'=\frac{c}{\text{Tr}_{F/F_0}(pc)}$, 
 \end{center}
 fulfills conditions (\ref{eqn 3.35}), (\ref{eqn 3.36}), and hence $\psi_{-1}|_{O_F}=c'\cdot\psi_F|_{O_F}=\psi_{q_F}$.
\end{lem}
\begin{proof}
 By the given condition 
 we have $\nu_{F}(c)=-1-d_{F/\bbQ_p}=-1-(e_{F/\bbQ_p}-1)=-e_{F/\bbQ_p}$. Then we can write 
 $c=\pi_{F}^{-e_{F/\bbQ_p}}u(c)=p^{-1}u(c)$ for some $u(c)\in U_F$ because $F/\bbQ_p$ is tamely ramified, hence $p=\pi_{F}^{e_{F/\bbQ_p}}$. 
Then we can write 
 $$\text{Tr}_{F/F_0}(pc)=p\cdot\text{Tr}_{F/F_0}(c)=p\cdot p^{-1}u_0(c)=u_0(c)\in U_{F_0}\subset U_{F},$$ 
 where $u_0(c)=\text{Tr}_{F/F_0}(u(c))$,
 hence $\nu_{F}(\text{Tr}_{F/F_0}(pc))=0$. Then the valuation of $c'$ is:
 \begin{center}
  $\nu_F(c')=\nu_F(\frac{c}{\text{Tr}_{F/F_0}(pc)})=\nu_F(c)-\nu_F(\text{Tr}_{F/F_{0}}(pc))$\\
  $=\nu_F(c)-0=\nu_{F}(c)=-1-d_{F/\bbQ_p}$.
 \end{center}
 Since $\text{Tr}_{F/F_0}(pc)=u_0(c)\in U_{F_0}$, we have 
 \begin{center}
  $\text{Tr}_{F/F_0}(c')=\text{Tr}_{F/F_0}(\frac{c}{\text{Tr}_{F/F_0}(pc)})
  =\frac{1}{\text{Tr}_{F/F_0}(pc)}\cdot\text{Tr}_{F/F_0}(c)=\frac{1}{p\cdot\text{Tr}_{F/F_0}(c)}\cdot\text{Tr}_{F/F_0}(c)=\frac{1}{p}$.
 \end{center}

Thus we observe that here $c'\in F^\times$ satisfies equations (\ref{eqn 3.35}) and (\ref{eqn 3.36}). 
Therefore, from Lemma \ref{Lemma 3.20} we can see that $\psi_{-1}|_{O_F}=c'\cdot\psi_{F}|_{O_F}$ is the canonical additive 
character of $k_F$.
\end{proof}

By Lemmas \ref{Lemma 3.20} and \ref{Lemma 3.21} we get many good (in the sense that
$\psi_{-1}|_{O_F}=c\cdot\psi_{F}|_{O_F}=\psi_{q_F}$)
elements $c$ which we will use in our next theorem to calculate $\lambda_{K/F}$, where 
$K/F$ is a tamely ramified quadratic extension.


\begin{proof}[{\bf Proof of Theorem \ref{Theorem 3.21}}]
From \cite{BH}, p. 190, part (2) of the Proposition, we have 
$$\lambda_{K/F}(\psi_{-1})=\lambda_{K/F}(c'\psi_{F})=\Delta_{K/F}(c')\cdot \lambda_{K/F}(\psi_F).$$
Since $\Delta_{K/F}$ is quadratic, we can write $\Delta_{K/F}=\Delta_{K/F}^{-1}$. So we obtain
\begin{equation*}
 \lambda_{K/F}(\psi_F)=\Delta_{K/F}(c')\cdot\lambda_{K/F}(\psi_{-1}).
\end{equation*}
Now we have to compute $\lambda_{K/F}(\psi_{-1})$, and which we do in the following:\\
  Since $[K:F]=2$, we have $\text{Ind}_{K/F}(1)=1_F\oplus\omega_{K/F}$.
 The conductor of $\omega_{K/F}$ is $1$ because $K/F$ is a tamely ramified quadratic extension, and hence $t=0$, so 
 $a(\omega_{K/F})=t+1=1$. Therefore, we can consider 
$\omega_{K/F}$ as a character of $F^\times/U_{F}^{1}$. So the restriction of $\omega_{K/F}$ to $U_{F}$, 
$\text{res}(\omega_{K/F}):=\omega_{K/F}|_{U_F}$, we may consider as the uniquely determined
character of $k_{F}^{\times}$ of order $2$.
Since $c'$ satisfies equations (\ref{eqn 3.35}), (\ref{eqn 3.36}), then from Lemma \ref{Lemma 3.21} we have 
$\psi_{-1}|_{O_F}=c'\cdot\psi_F|_{O_F}=\psi_{q_F}$, and this is the canonical
character of $k_F$. Then from equation (\ref{eqn 3.33}) we can write
\begin{align*}
 \lambda_{K/F}(\psi_{-1})
 &=q_{F}^{-\frac{1}{2}}\sum_{x\in k_{F}^{\times}}\text{res}(\omega_{K/F})(x)\psi_{q_F}(x)\\
 &=q_{F}^{-\frac{1}{2}}\cdot G(\text{res}(\omega_{K/F}),\psi_{q_F}).
\end{align*}
Moreover, by Theorem \ref{Theorem 2.7} we have 
\begin{equation}
 G(\text{res}(\omega_{K/F}),\psi_{q_F})=\begin{cases}
                    (-1)^{s-1}q_{F}^{\frac{1}{2}} & \text{if $p\equiv 1\pmod{4}$}\\
                    (-1)^{s-1}i^{s}q_{F}^{\frac{1}{2}} & \text{if $p\equiv 3\pmod{4}$}.
                   \end{cases}
\end{equation}
By using the classical quadratic Gauss sum we obtain
\begin{equation}
 \lambda_{K/F}(\psi_{-1})=\begin{cases}
                    (-1)^{s-1} & \text{if $p\equiv 1\pmod{4}$}\\
                    (-1)^{s-1}i^{s} & \text{if $p\equiv 3\pmod{4}$}.
                   \end{cases}
\end{equation}

 We also can write 
$\Delta_{K/F}=\det(\text{Ind}_{K/F}(1))=\det(1_F\oplus \omega_{K/F})=\omega_{K/F}.$
So we have 
 $$\Delta_{K/F}(\pi_F)=\omega_{K/F}(\pi_F)=1,$$
because $\pi_F\in N_{K/F}(K^\times)$.

Under the assumption of the Theorem \ref{Theorem 3.21} we have $\pi_F\in N_{K/F}(K^\times)$, $\Delta_{K/F}=\omega_{K/F}$ and 
$c'=\frac{c}{\text{Tr}_{F/F_0}(pc)}$, where $c\in F^\times$ with $\nu_F(c)=-1-d_{F/\bbQ_p}$. Then we can write 
\begin{align*}
 \Delta_{K/F}(c')
 =\omega_{K/F}(c')
 &=\omega_{K/F}\left(\frac{c}{\text{Tr}_{F/F_0}(pc)}\right)\\
 &=\omega_{K/F}\left(\frac{\pi_{F}^{-e_{F/\bbQ_p}}u(c)}{u_0(c)}\right),
 \quad\text{where $c=\pi_{F}^{-e_{F/\bbQ_p}}u(c)$, $\text{Tr}_{F/F_0}(pc)=u_0(c)\in U_{F_0}$}\\
 &=\omega_{K/F}(\pi_{F}^{-e_{F/\bbQ_p}})\omega_{K/F}(v),\quad\text{where $v=\frac{u(c)}{u_0(c)}\in U_F$}\\
 &=\omega_{K/F}(x)\\
 &=\begin{cases}
                  1 & \text{when $x$ is a square element in $k_{F}^{\times}$}\\
                  -1 & \text{when $x$ is not a square element in $k_{F}^{\times}$},
                 \end{cases}
\end{align*}
where $v=xy$, with $x=x(\omega_{K/F},c)\in U_{F}/U_{F}^{1}$, and $y\in U_{F}^{1}$.

In particular, if we choose $c$ such a way that $u(c)=1$, i.e., $c=\pi_{F}^{-1-d_{F/\bbQ_p}}$, then 
we have 
$\Delta_{K/F}(c')=\Delta_{K/F}(\text{Tr}_{F/F_0}(pc)).$
Since $\text{Tr}_{F/F_0}(pc)\in O_{F_0}$ is a unit and $\Delta_{K/F}=\omega_{K/F}$ induces the quadratic character of 
$k_{F}^{\times}=k_{F_0}^{\times}$, then for this particular choice of $c$ we obtain
\begin{equation*}
 \Delta_{K/F}(c')=\begin{cases}
                   1 & \text{if $\overline{\text{Tr}_{F/F_0}(pc)}$ is a square in $k_{F_0}^{\times}$}\\
                   -1 & \text{if $\overline{\text{Tr}_{F/F_0}(pc)}$ is not a square in $k_{F_0}^{\times}$}.\\
                  \end{cases}
\end{equation*}

\end{proof}

\subsection{\textbf{Computation of $\lambda_{K/F}$, where $K/F$ is a wildly ramified extension}}

In the case $p=2$, the square class group of $F$, i.e., $F^\times/{F^\times}^2$ can be very large 
(cf. Theorem 2.29 on p. 165 of \cite{TYM}), 
so we can have many quadratic characters but they are 
wildly ramified, not tame. In Remark \ref{Remark 1.2}, we mention the current status of 
the this wildly ramified quadratic case.

\begin{proof}[{\bf Proof of Theorem \ref{Theorem 3.26}}]

Let $G=\mathrm{Gal}(K/F)$. From Theorem 2.29 on p. 165 of \cite{TYM}, if $F/\bbQ_2$, 
we have $|F^\times/{F^\times}^2|=2^m,\,(m\ge 3)$ and  the $2$-rank of $G$ (i.e., the dimension of $G/G^2$ as a vector 
space over
$\bbF_2$) $\text{rk}_2(G)\ne 1$
 and $G$ is not metacyclic. Then from Bruno Kahn's result, the second Stiefel-Whitney class $s_2(\text{Ind}_{K/F}(1))=0$ 
 (cf. Theorem 1 of \cite{BK}). Since $s_2(\text{Ind}_{K/F}(1))=0$, the Deligne constant $c(\text{Ind}_{K/F}(1))=1$
 (cf. Theorem 3 on p. 129 of \cite{JT1}). Again since here $\text{rk}_2(G)\ne 1$ and $G$ is not metacyclic, 
 we have $\Delta_{K/F}\equiv 1$.
 Therefore, we can conclude that 
 $$\lambda_{K/F}(\psi)=c(\text{Ind}_{K/F}(1))\cdot W(\Delta_{K/F},\psi)=1,$$
 where $\psi$ is a nontrivial additive character of $F$.
\end{proof}


\begin{exm}[{\bf Computation of $\lambda_{K/\bbQ_2}$, where $K/\bbQ_2$ is a quadratic extension}]\label{Example wild}

In this case, we have (cf. pp. 60-63 of \cite{SABT}):
$$\lambda_{\bbQ_2(\sqrt{5})/\bbQ_2}=1, \lambda_{\bbQ_2(\sqrt{-1})/\bbQ_2}=i, \lambda_{\bbQ_2(\sqrt{-5})/\bbQ_2}=i,
\lambda_{\bbQ_2(\sqrt{2})/\bbQ_2}=1,$$
$$\lambda_{\bbQ_2(\sqrt{10})/\bbQ_2}=-1, 
\lambda_{\bbQ_2(\sqrt{-2})/\bbQ_2}=i, \lambda_{\bbQ_2(\sqrt{-10})/\bbQ_2}=-i.$$
\end{exm}

\vspace{1cm}

\textbf{Acknowledgements.} I would like to thank Prof. E.-W. Zink, Humboldt University, Berlin
for suggesting this problem and his constant 
valuable advice and comments. I
express my gratitude to the referee for his/her valuable comments and suggestions for the improvement of the paper.

\end{document}